\documentclass[12pt]{article} 

\usepackage{amsmath}
\usepackage{amsthm}
\usepackage{amsfonts}
\usepackage{mathrsfs}
\usepackage{stmaryrd}
\usepackage{setspace}
\usepackage{fullpage}
\usepackage{amssymb}
\usepackage{breqn}
\usepackage{enumitem}
\usepackage{bbold} 
\usepackage{authblk}
\usepackage{comment}
\usepackage{hyperref}
\usepackage{pgf,tikz}
\usepackage{graphicx}
\usepackage{subcaption}

\newtheorem*{thm*}{Theorem}
\newtheorem{thm}{Theorem}[section]

\newtheorem{lemma}[thm]{Lemma}
\newtheorem{conjecture}[thm]{Conjecture}

\newtheorem*{prop*}{Proposition}

\newtheorem{clm}[thm]{Claim}

\newcommand\ex{\ensuremath{\mathrm{ex}}}

\newcommand\cH{{\mathcal H}}

\newcommand\cN{{\mathcal N}}

\newcommand{\ignore}[1]{}

\title{On non-degenerate Berge-Turán problems}
\author{Dániel Gerbner\\\small Alfr\'ed R\'enyi Institute of Mathematics\\
\small \texttt{gerbner.daniel@renyi.hu}}

\date{}

\begin{document}

\maketitle

\begin{abstract}
Given a hypergraph $\cH$ and a graph $G$, we say that $\cH$ is a \textit{Berge}-$G$ if there is a bijection between the hyperedges of $\cH$ and the edges of $G$ such that each hyperedge contains its image. We denote by $\ex_k(n,\textup{Berge-}F)$ the largest number of hyperedges in a $k$-uniform Berge-$F$-free graph. Let $\ex(n,H,F)$ denote the largest number of copies of $H$ in $n$-vertex $F$-free graphs. It is known that $\ex(n,K_k,F)\le \ex_k(n,\textup{Berge-}F)\le \ex(n,K_k,F)+\ex(n,F)$, thus if $\chi(F)>r$, then $\ex_k(n,\textup{Berge-}F)=(1+o(1)) \ex(n,K_k,F)$. We conjecture that $\ex_k(n,\textup{Berge-}F)=\ex(n,K_k,F)$ in this case. We prove this conjecture in several instances, including the cases $k=3$ and $k=4$. We prove the general bound $\ex_k(n,\textup{Berge-}F)= \ex(n,K_k,F)+O(1)$.
\end{abstract}




\section{Introduction}

Given a hypergraph $\cH$ and a graph $G$, we say that $\cH$ is a \textit{Berge copy} of $G$ (in short: a Berge-$G$) if there is a bijection between the hyperedges of $\cH$ and the edges of $G$ such that each hyperedge contains its image.

Berge hypergraphs were introduced by Gerbner and Palmer \cite{gp1} as a generalization of the notion of hypergraph cycles due to Berge.

A closely connected area is that of generalized Tur\'an problems. Given graphs $H$ and $G$, we let $\cN(H,G)$ denote the number of copies of $H$ in $G$. Let $\ex(n,H,F):=\max \{\cN(H,G): \text{ $G$ is an $n$-vertex $F$-free graph}\}$. The sytematic study of this topic was initiated by Alon and Shikhelman \cite{ALS2016} after several sporadic results.

The connection between Berge-Tur\'an problems and generalized Tur\'an problems was established by Gerbner and Palmer \cite{gp2}, who showed that $\ex(n,K_k,F)\le \ex_k(n,\textup{Berge-}F)\le \ex(n,K_k,F)+\ex(n,F)$. The upper bound was strengthened by F\"uredi, Kostochka, and Luo~\cite{fkl} and independently by Gerbner, Methuku and Palmer \cite{gmp}. To state this result, we need some definition.

A \textit{blue-red graph} $G$ is a graph with each edge colored blue or red. We denote by $G_{blue}$ the subgraph consisting of the blue edges and by $G_{red}$ the subgraph consisting of the red edges.
We say that a blue-red graph $G$ is $F$-free if $G$ does not contain $F$ (here we do not care about the colors). Given an integer $k\ge 3$, let $g(G):=\cN(K_k,G_{blue})+|E(G_{red})|$. Let $\ex^{{\mathrm col}}(n,F):=\max\{g(G): \text{ $G$ is an $n$-vertex $F$-free graph}\}$.

\begin{lemma}[\cite{fkl},\cite{gmp}]\label{celeb2}
    For any $F$ we have $\ex_k(n,\textup{Berge-}F)\le\ex^{{\mathrm col}}(n,F)$.
\end{lemma}

A hypergraph Tur\'an problem is called degenerate if the order of magnitude of the extremal function is smaller than the largest possible, i.e. smaller than $n^k$ in our case. By the above bounds, $\ex_k(n,\textup{Berge-}F)=o(n^k)$ if and only if $\ex(n,K_k,F)=o(n^k)$, which happens if and only if $\chi(F)\le k$ by a result of Alon and Shikhelman \cite{ALS2016}. Another result of Alon and Shikhelman \cite{ALS2016} shows that if $\chi(F)=r+1>k$, then $\ex(n,K_k,F)=(1+o(1))\cN(K_k,T(n,r))=(1+o(1))\binom{r}{k}\left(\frac{n}{r}\right)^k$.

In the non-degenerate case, for $k\ge 3$ we have that $\ex(n,F)=O(n^2)=o(\ex(n,K_k,F))$, thus $\ex_k(n,\textup{Berge-}F)=(1+o(1)) \ex(n,K_k,F)$.
We believe that a stronger connection also holds.

\begin{conjecture}\label{sejt} If $\chi(F)>k$ and $n$ is sufficiently large, then $\ex_k(n,\textup{Berge-}F)=\ex(n,K_k,F)$.
\end{conjecture}

The above conjecture is known to hold in the case $F$ has a color-critical edge (an edge whose deletion decreases the chromatic number). The $k$-uniform \textit{expansion} $F^{+k}$ of a graph $F$ is the specific $k$-uniform Berge copy that contains the most vertices, i.e., the $k-2$ vertices added to each edge of $F$ are distinct for different edges, and distinct from the vertices of $F$. Pikhurko \cite{pikhu} showed that for $r\ge k$, the Tur\'an number of $K_{r+1}^{+k}$ is equal to $\cN(K_k,T(n,r))$ if $n$ is sufficiently large. According to the survey \cite{mubver} on expansions, Alon and Pikhurko observed that Pikhurko's proof generalizes to the case $F$ is an $(r+1)$-chromatic graph with a color-critical edge. A simpler proof for the Berge case can be found in \cite{ger}.

In general, the above observations imply that $\ex_k(n,\textup{Berge-}F)= \ex(n,K_k,F)+O(n^2)$. This was improved to $\ex_k(n,\textup{Berge-}F)= \ex(n,K_k,F)+o(n^2)$ in \cite{ge}. We further improve this bound in our next result.

\begin{thm}\label{konst}
$\ex_k(n,\textup{Berge-}F)= \ex(n,K_k,F)+O(1)$.
\end{thm}

We show that Conjecture \ref{sejt} holds if $F$ contains a color-critical vertex (a vertex whose deletion decreases the chromatic number).

\begin{thm}\label{cc} 
Let $\chi(F)>k$ and assume that $F$ contains a color-critical vertex. Then for sufficiently large $n$ we have
$\ex_k(n,\textup{Berge-}F)=\ex(n,K_k,F)$.
\end{thm}

We show that Conjecture \ref{sejt} holds in the 3- and 4-uniform case. Furthermore, it holds in any uniformity if the chromatic number of $F$ is sufficiently large.

\begin{thm}\label{kisk} 
\textbf{(i)} Let $\chi(F)>k$ and $k\le 4$. Then
$\ex_k(n,\textup{Berge-}F)=\ex(n,K_k,F)$ for sufficiently large $n$.

\textbf{(i)} Let us fix $k$ and $r$ be sufficiently large. If $\chi(F)=r+1$, then $\ex_k(n,\textup{Berge-}F)=\ex(n,K_k,F)$ for sufficiently large $n$.
\end{thm}

Recall that if $\chi(F)>k$, then the asymptotics of $\ex(n,K_k,F)$ is known, thus the asymptotics of $\ex_k(n,\textup{Berge-}F)$ is known.
Even if Conjecture \ref{sejt} is true, it only improves the asymptotic result to an exact result in the few cases when $\ex(n,K_k,F)$ is known. Besides the case where $F$ has a color-critical edge, we are aware only of the following results. Let $2K_{r+1}$ denote two vertex-disjoint copies of $K_{r+1}$ and $B_{r,1}$ denote two copies of $K_{r+1}$ sharing exactly one vertex. Gerbner and Patk\'os \cite{gerpat} determined $\ex(n,K_k,2K_{r+1})$ and $\ex(n,K_k,B_{r+1,1})$. The first of these results was extended by Gerbner \cite{ge2} to $\ex(n,K_k,F)$ in the case each component of $F$ either has chromatic number $r+1$ and contains a color-critical edge, or has chromatic number at most $r$. Gerbner \cite{ge2} also determined $\ex(n,K_k,Q_{r+1})$ for a class of graphs $Q_r$ that we do not define here and most values of $k$.

For the Berge copies of the graphs mentioned above, we can show that Conjecture \ref{sejt} holds. In fact, $B_{r+1,1}$ and $Q_r$ each has a color-critical vertex, thus we already dealt with them in Theorem \ref{cc}. Let $K_i+T(n-i,r)$ denote the graph we obtain by adding $i$ vertices to $T(n-i,r)$ and joining them to every vertex. 

\begin{thm}\label{blab}
Let us assume that $F$ consists of $s$ components with chromatic number $r+1$, each with a color-critical edge, and any number of components with chromatic number at most $r$. Then $\ex_k(n,\textup{ Berge-}F)=\cN(K_k,K_{s-1}+T(n-s+1,r))$.
\end{thm}

To prove the above theorems, we use the following results on the structure of the extremal graphs that are interesting on their own. Let us denote by $\sigma(F)$ the smallest possible order of a color class in a $\chi(F)$-coloring of $F$.

\begin{thm}\label{strucc} Let $\chi(F)=r+1>k$ and $G$ be an $n$-vertex $F$-free blue-red graph with $g(G)=\ex^{{\mathrm col}}(n,F)$. Then the followings hold.

 \textbf{(i)} For every vertex $u$ of $G$ we have that the number of blue $k$-cliques plus the number of red edges containing $u$ is at least $(1+o(1))\binom{r-1}{k-1}(\frac{n}{r})^{k-1}$.

  \textbf{(ii)}  Let $\varepsilon>0$ be sufficiently small. Then there exist an $r$-partition of $V(G)$ to $A_1,\dots, A_r$, a constant $K=K(F,\varepsilon)$ and a set $B$ of at most $rK(\sigma(F)-1)$ vertices such that the followings hold. For each $i$ we have$|A_i|=(1-o(1))n/r$, each red edge is between two elements of $B$, every vertex of $B$ is adjacent to 
  at least $\varepsilon n$ vertices in each part and to at least $cn$ vertices in all but one parts for some constant $c=c(F)$. Furthermore, every vertex of $A_i\setminus B$ is adjacent to at most $\varepsilon n$ vertices in $A_i$ and all but at most $\varepsilon (2r^k+1) n$ vertices in $A_j$ with $j\neq i$.

 \textbf{(iii)} Let $\cH$ be an $n$-vertex $k$-uniform Berge-$F$-free hypergraph with $\ex_k(n,\textup{Berge-}F)$ hyperedges. Then every vertex of $\cH$ is contained in at least $(1+o(1))\binom{r-1}{k-1}(\frac{n}{r})^{k-1}$ hyperedges.

\end{thm}

\section{Proofs}

We will use the following stability result due to Ma and Qiu \cite{mq}.

\begin{thm}[Ma, Qiu \cite{mq}]\label{maqi}
    Let $\chi(F)>k$ and let $G$ be an $n$-vertex $F$-free graph that contains $\ex(n,K_k,F)-o(n^k)$ copies of $K_k$. Then $G$ can be turned into $T(n,r)$ by adding and removing $o(n^2)$ edges.
\end{thm}

Let us start with the proof of Theorem \ref{strucc}, that we restate here for convenience.

\begin{thm*}  Let $\chi(F)=r+1>k$ and $G$ be an $n$-vertex $F$-free blue-red graph with $g(G)=\ex^{{\mathrm col}}(n,F)$. Then the followings hold.

 \textbf{(i)} For every vertex $u$ of $G$ we have that the number of blue $k$-cliques plus the number of red edges containing $u$ is at least $(1+o(1))\binom{r-1}{k-1}(\frac{n}{r})^{k-1}$.

  \textbf{(ii)}  Let $\varepsilon>0$ be sufficiently small. Then there exist an $r$-partition of $V(G)$ to $A_1,\dots, A_r$, a constant $K=K(F,\varepsilon)$ and a set $B$ of at most $rK(\sigma(F)-1)$ vertices such that the followings hold. For each $i$ we have$|A_i|=(1-o(1))n/r$, each red edge is between two elements of $B$, every vertex of $B$ is adjacent to 
  at least $\varepsilon n$ vertices in each part and to at least $cn$ vertices in all but one parts for some constant $c=c(F)$. Furthermore, every vertex of $A_i\setminus B$ is adjacent to at most $\varepsilon n$ vertices in $A_i$ and all but at most $\varepsilon (2r^k+1) n$ vertices in $A_j$ with $j\neq i$.

 \textbf{(iii)} Let $\cH$ be an $n$-vertex $k$-uniform Berge-$F$-free hypergraph with $\ex_k(n,\textup{Berge-}F)$ hyperedges. Then every vertex of $\cH$ is contained in at least $(1+o(1))\binom{r-1}{k-1}(\frac{n}{r})^{k-1}$ hyperedges.
\end{thm*}

We note that the analogous results for $\ex(n,K_k,F)$ can be found in \cite{mq}. Generalizations to some other graphs in place of $K_k$ can be found in \cite{gerb} for \textbf{(i)} and in \cite{ge2} for \textbf{(ii)}. Our proof follows the proofs in \cite{gerb} and \cite{ge2}.

\begin{proof} Observe that $G$ contains at least $\ex(n,K_k,F)-\ex(n,F)$ blue copies of $K_k$, thus $G_{blue}$ can be transformed to a complete $r$-partite graph by adding and removing $o(n^2)$ edges by Theorem \ref{maqi}. Note that there may be several different such complete $r$-partite graphs on the vertex set $V(G)$ that can be obtained this way, we pick one with the smallest number of edges inside the parts and denote it by $G'$. It is easy to see that each part has order $(1-o(1))n/r$, otherwise the number of blue cliques is at most $\binom{r}{k}\left(\frac{n}{r}\right)^k-\Theta(n^k)$. Let $A_1,\dots,A_r$ denote the parts and let $f(v)$ denote the number of red edges and blue $k$-cliques incident to $v$ that are removed this way. Then we have $\sum _{v\in V(G)}f(v)=o(n^k)$. Consider a set $S$ of $|V(F)|$ vertices in $A_1$ such that $\sum _{v\in S}f(v)$ is minimal. Then by averaging $\sum _{v\in S}f(v)\le \frac{|S|}{|V_1|}\sum _{v\in V_1}f(v)=o(n^{k-1})$.

Let us consider blue $k$-cliques and red edges that contain exactly one vertex $s$ of $S$, and the other vertices are in the common neighborhood of $S$ in $G$. Let
$d_G(k,S)$ denote the number of such blue $k$-cliques and red edges. Observe that each vertex of $S$ is in $\frac{d_G(k,S)}{|S|}$ such blue $k$-cliques and red edges. Clearly $\frac{d_{T(n,r)}(k,S)}{|S|}=(1+o(1))\binom{r-1}{k-1}(\frac{n}{r})^{k-1}$.

Let $x$ denote the number of blue $k$-cliques and red edges that contain $u$ and a vertex from $S$, then $x=O(n^{k-2})$. Now we apply a variant of Zykov's symmetrization \cite{zykov}. If $d_G(k,u)<\frac{d_G(k,S)}{|S|}-x$, then we remove the edges incident to $u$ from $G$. Then for every vertex $v$ that is connected to each vertex of $S$ with a blue edge, we connect $u$ to $v$ with a blue edge. For every vertex $w$ that is connected to each vertex of $S$ with a red edge, we connect $u$ to $w$ with a red edge. This way we do not create any copy of $F$, as the copy should contain $u$, but $u$ could be replaced by any vertex of $S$ that is not already in the copy, to create a copy of $F$ in $G$. We removed $d_G(k,u)$ blue $k$-cliqes and red edges, but added at least $\frac{d_G(H,S)}{|S|}-x$ blue $k$-cliqes and red edges, a contradiction. 

Therefore, we have that the blue $k$-cliques plus the red edges containing $u$ is at least \[\frac{d_G(k,S)}{|S|}-x\ge \frac{d_{T(n,r)}(k,S)}{|S|}-\sum_{v\in S}f(v)-x=\frac{d_{T(n,r)}(k,S)}{|S|}-o(n^{k-1}).\]

This completes the proof of \textbf{(i)}. 

The proof of \textbf{(iii)} is similar. We pick $S$ the same way, but instead of blue $k$-cliques and red edges, we count the hyperedges containing $u$, let $d_{\cH}(k,u)$ denote their number. Let $y$ denote the number of hyperedges that contain $u$ and a vertex from $S$. If $d_\cH(k,u)<\frac{d_\cH(k,S)}{|S|}-x$, then we remove the hyperedges containing $u$ and for every hyperedge $H$ that contains exactly one vertex $v\in S$, we add $(H\setminus \{v\})\cup \{u\}$ as a hyperedge. Then the same reasoning as above completes the proof of \textbf{(iii)}.

Let $B$ denote the set of vertices that are adjacent to at least $\varepsilon n$ vertices in their part $A_i$. Note that by the choice of $G'$, vertices of $B$ are incident to at least $\varepsilon n$ vertices in each other part. Let $B_i=B\cap A_i$.

\begin{clm} There is a $K$ depending on $\varepsilon$ and $F$ such that
$|B|\le K(\sigma(F)-1)$.
\end{clm}

The analogous claim for uncolored graphs $G_0$ with $\ex(n,K_k,F)$ copies of $K_k$ is in \cite{mq}. However, the proof of that claim does not use that $G_0$ is extremal, only that $G_0$ contains $\ex(n,K_k,F)-o(n^k)$ copies of $K_k$. As this holds for $G$ as well, the claim follows.


Consider now the set $U_i$ of vertices $v$ such that $v\in A_i\setminus B$ is adjacent to less than $|A_j|-\varepsilon(2r^k+1) n$ vertices of some $A_j$. As there are $o(n^2)$ edges missing between parts, we have that $|U_i|=o(n)$. 

\begin{clm}
    For each $i$ we have that $U_i=\emptyset$.
\end{clm}

\begin{proof}[Proof of Claim]
Let $V_i=A_i\setminus (B_i\cap U_i)$, then we have that $|V_i|\ge |A_i|-\varepsilon n$.

Let us delete the edges from each $v\in U_i$ to $A_i$ and connect $v$ to each vertex of each $V_j$, $j\neq i$ with a blue edge. We claim that the resulting graph $G''$ is $F$-free. Indeed, consider a copy $F_0$ of $F$ with the smallest number of vertices in $U_i$. Clearly $F_0$ contains a vertex $v\in U_i$, as all the new edges are incident to such a vertex.  Let $Q$ be the set of vertices in $F_0$ that are adjacent to $v$ in $G'$. They are each from $\cup_{j\neq i}V_j$. Their common neighborhood in $V_i$ is of order $\frac{n}{r}-o(n)$.
Therefore, at least one of the common neighbors is not in $F_0$, thus we can replace $v$ with that vertex to obtain another copy of $F$ with less vertices from $\cup_{i=1}^r U_i$, a contradiction.

We deleted at most $\varepsilon n^{k-1}$ blue $k$-cliques and red edges for each vertex $v\in U_i$, since they each contain one of the less than $\epsilon n$ edges incident to $v$ inside $U_i$. We claim that we added more than $\varepsilon n^{k-1}$ blue $k$-cliques. We consider only those blue $k$-cliques that contain $v$, a new neighbor of $v$ in $V_j$ with $j\neq i$, and $k-2$ other vertices from other sets $V_\ell$. We have at least $2r^k\varepsilon n$ choices for the neighbor and at least $n/r-\varepsilon n$ choices for the other vertices. If $\varepsilon$ is sufficiently small, then indeed, we obtain more than $\varepsilon n$ new blue $k$-cliques, thus $g(G'')>g(G')$, a contradiction unless $U_i$ is empty.
\end{proof}

Now we show that there is a constant $c=c(F)$ such that each vertex is adjacent to at least $cn$ vertices in all but one parts. Assume that $v$ is adjacent to less than $cn$ vertices in $A_1$ and in $A_2$. Then the number of blue cliques containing $v$ is at most $\binom{r-2}{k-1}(\frac{n}{r})^{k-1}+\binom{r-2}{k-2}(\frac{n}{r})^{k-2}cn+\binom{r-2}{k-3}(\frac{n}{r})^{k-3}cn^2$, contradiction to \textbf{(i)} if $c$ is small enough.

It is left to show that each red edge is between vertices in $B$. Assume that $u\not\in B$ and $uv$ is a red edge. Let us change its color to blue. We will find more than one new blue $k$-clique greedily. We can assume without loss of generality that $u\in A_1$ and $v$ is in either $A_1$ or in $A_2$. Let us observe that $u$ and $v$ have at least $cn-\varepsilon(2r^k+2) n$ common neighbors in $G_{blue}$ inside $V_3$, we pick one of them. These three vertices have at least $cn-2\varepsilon(2r^k+2) n$ common neighbors in $G_{blue}$ inside $V_4$, we pick one of them, and so one. We can pick $k$ vertices if $cn-(k-2)\varepsilon(2r^k+2) n>0$, which holds if $\varepsilon$ is small enough. Clearly we can pick more than one blue $k$-clique this way, completing the proof of \textbf{(ii)}.
\end{proof}

Theorem \ref{konst} is easily implied by \textbf{(ii)} of Theorem \ref{strucc}, since in an $F$-free $n$-vertex blue-red graph, the number of blue $k$-cliques is at most $\ex(n,K_k,F)$, while the number of red edges inside $B$ is $O(1)$. Theorem \ref{cc} is also implied by  \textbf{(ii)} of Theorem \ref{strucc}, since a color-critical vertex means that $\sigma(F)=1$, thus $|B|=0$, hence there are no red edges.

Let us continue with the proof of Theorem \ref{kisk}. Recall that it states that if $\chi(F)>k$ and $k\le 4$ or if $\chi(F)$ is sufficiently large, then Conjecture \ref{sejt} holds.

\begin{proof}[Proof of Theorem \ref{kisk}] Let $\chi(F)=r+1$.
We will use Lemma \ref{celeb2}. Let $G$ be a blue-red $F$-free graph with $g(G)=\ex^{{\mathrm col}}(n,F)$. Assume that there is a red edge $uv$ in $G$ and apply now \textbf{(ii)} of Theorem \ref{strucc}. We obtain a partition of $V(G)$ to $A_1,\dots, A_r$ with $|A_i|=(1+o(1))n/r$ such that there are $o(n)$ edges inside parts, and there is a set $B$ of vertices with $|B|=o(n)$ such that
each vertex outside $B$ is adjacent to all but $o(n)$ vertices in each other part.

Assume that $u$ and $v$ have $\Omega(n)$ common neighbors in at least $k-2$ of the sets $A_1,\dots,A_r$, say $A_1,\dots,A_{k-2}$. Then at least $\Omega(n)$ of those vertices are not in $B$, we will use only those vertices. We pick a common neighbor in $A_1\setminus B$, then it has $\Omega(n)$ common neighbor with $u$ and $v$ in $A_2$. Therefore, we can pick a common neighbor in $A_2\setminus B$, and so on. The resulting cliques do not contain any vertex of $B$, thus by turning $uv$ blue, we obtain multiple blue $k$-cliques, thus $g(G)$ increases, a contradiction.

We obtained that $u$ and $v$ have $\Omega(n)$ common neighbors in at most $k-3$ of the sets $A_i$, say $A_1,\dots, A_{k-3}$. In the remaining $r-k+3$ sets $A_i$, they have $o(n)$ common neighbors, thus at least one of them, say $u$ has at most $(1+o(1))(r-k+3)n/2r$ neighbors in $A_{k-2},\dots, A_r$. Consider now the number of blue $k$-cliques containing $u$. There are $o(n^{k-1})$ blue $k$-cliques that contain $u$ and an edge inside an $A_i$ that is not incident to $u$. Therefore, we can focus on those blue $k$-cliques that contain $u$, and the other $k-1$ vertices are from different parts. 

Let $K$ be such a blue $k$-clique and assume that $K$ contains $i$ vertices from $A_1,\dots,A_{k-3}$. There are at most $(1+o(1))\binom{k-3}{i}\left(\frac{n}{r}\right)^{i}$ ways to pick such an $i$-set. For the remaining $k-1-i$ vertices of $K$, we have to pick one neighbor of $u$ from $k-1-i$ of the remaining $r-k+3$ sets, and in total $u$ has $(1+o(1))(r-k+3)n/2r$ neighbors in those sets. Then the number of $(k-1-i)$-cliques is at most $\ex((1+o(1))(r-k+3)n/2r,K_{k-1-i},K_{r-k+4})$. A theorem of Zykov \cite{zykov} states that $\ex(n,K_s,K_t)=\cN(K_s,T(n,t-1))=(1+o(1))\binom{t-1}{s}\left(\frac{n}{t-1}\right)^{s}$, thus 
there are at most $(1+o(1))\binom{r-k+3}{k-i-1}\left(\frac{n}{2r}\right)^{k-1-i}$ ways to pick the $(k-1-i)$-clique.

We apply \textbf{(i)} of Theorem \ref{strucc}, thus we know that each vertex $v$ is in at least $(1+o(1))\binom{r-1}{k-1}(\frac{n}{r})^{k-1}$ blue $k$-cliques. Therefore, 
\[(1+o(1))\sum_{i=0}^{k-3}\binom{k-3}{i}\left(\frac{n}{r}\right)^{i}\binom{r-k+3}{k-i-1}\left(\frac{n}{2r}\right)^{k-1-i}\ge (1+o(1))\binom{r-1}{k-1}\left(\frac{n}{r}\right)^{k-1}.\]

This holds only if

\begin{equation}\label{equ}\sum_{i=0}^{k-3}\binom{k-3}{i}\binom{r-k+3}{k-i-1}\left(\frac{1}{2}\right)^{k-1-i}\ge \binom{r-1}{k-1}.\end{equation}

If $k=3$, then $i=0$ and $\binom{r}{k-1}/4\ge \binom{r-1}{k-1}$, a contradiction.

If $k=4$, then (\ref{equ}) gives $\binom{r-1}{3}\left(\frac{1}{2}\right)^{3}+\binom{r-1}{2}\left(\frac{1}{2}\right)^{2}\ge \binom{r-1}{3}$. If $r\ge 6$, then $\binom{r-1}{2}\le \binom{r-1}{3}$, thus $\binom{r-1}{3}\left(\frac{1}{2}\right)^{3}+\binom{r-1}{2}\left(\frac{1}{2}\right)^{2}\le \binom{r-1}{3}\left(\frac{1}{8}+\frac{1}{4}\right)<\binom{r-1}{3}$, a contradiction. If $r=5$ or $r=4$, then one can easily obtain a contradiction as well. This completes the proof of \textbf{(i)}.


There are several other pairs $(k,r)$ when we could obtain a contradiction a similar way. However, if $k=r$, the left hand side has a term $\binom{k-3}{k-4}/8$. If $k\ge 11$, then this term alone is larger than the right hand side, thus we do not have a contradiction in general. In fact, one can easily see that for $k=r=5$ we do not obtain any contradiction. On the other hand, if $k$ is fixed and $r$ grows, there is only one term on the left hand side of (\ref{equ}) of order $r^{k-1}$, and it is $r^{k-1}/2^{k-1}(k-1)!$. Since the leading term on the right hand side is $r^{k-1}/(k-1)!$, we obtain a contradiction for $r$ large enough, proving \textbf{(ii)}.
\end{proof}

Let us continue with the proof of Theorem \ref{blab} that we restate here for convenience.

\begin{thm*}
    \textbf{(i)}
Let us assume that $F$ consists of $s$ components with chromatic number $r+1$, each with a color-critical edge, and any number of components with chromatic number at most $r$. Then $\ex_k(n,\textup{ Berge-}F)=\cN(K_k,K_{s-1}+T(n-s+1,r))$.

    \textbf{(ii)} $\ex_k(n,\textup{ Berge-}B_{r+1,1})=\cN(K_k,T^+(n,r))$.
\end{thm*}

The corresponding generalized Tur\'an results are proved in \cite{ge2} and we will extend the proofs from there. We omit some details. We remark that in \cite{ge2}, the proof of the statement $\ex(n,K_k,F)=\cN(K_k,K_{s-1}+T(n-s+1,r))$ shows a bit more: if an $n$-vertex $F$-free graph $G$ is not a subgraph of $K_{s-1}+T(n-s+1,r)$, then it contains $\cN(K_k,K_{s-1}+T(n-s+1,r))-\Omega(n^{k-1})$ copies of $K_k$. This immediately implies for us that $G_{blue}$ is a subgraph of $K_{s-1}+T(n-s+1,r)$. Changing any blue edge in $K_{s-1}+T(n-s+1,r)$ to red destroys $\Theta(n^{k-2})$ copies of $K_k$, thus it decreases $g(G)$. This gives an alternative proof of \textbf{(i)} of Theorem \ref{blab}.

\begin{proof} We start with proving \textbf{(i)}.
    Let $G$ be a blue-red $F$-free graph with $g(G)=\ex^{{\mathrm col}}(n,F)$. We apply \textbf{(ii)} of Theorem \ref{strucc}. Assume first that there are $s$ independent edges $u_1v_1,\dots, u_sv_s$ inside the parts such that for each $i$, at least one of $u_i$ and $v_i$ are not in $B$. Observe that $u_i$ and $v_i$ have $\Omega(n)$ common neighbors in each part besides the one containing them. Using this, we can easily extend each edge to an $(r+1)$-chromatic component of $F$, where $u_iv_i$ plays the role of a color-critical edge. We can also find the other components to obtain a copy of $F$ in $G$, a contradiction.

    If $|B|\ge s$, then we can find $s$ distinct vertices among their neighbors not in $B$, resulting in the contradiction. By similar reasoning, there are no $s-|B|$ independent edges inside parts but outside $B$. Therefore, the edges inside parts that are not incident to any vertex of $B$ form at most $s-1-|B|$ stars plus $O(1)$ further edges. Since the vertices outside $B$ are incident to $o(n)$ edges inside parts, there are $o(n^{k-1})$ $k$-cliques containing such a vertex. This implies that deleting all the edges inside parts that are not incident to $B$, we lose $o(n^{|V(H)|-1})$ copies of $H$.
If $|B|<s-1$, then we can add a vertex to $B$ creating $\Theta(n^{|V(H)|-1})$ copies of $H$, a contradiction. We obtained that $|B|=s-1$ and then there is no edge inside parts but outside $B$. This implies that $G$ is a subgraph of $K_{s-1}+T(n-s+1,r)$, completing the proof.
\end{proof}

\bigskip

\textbf{Funding}: Research supported by the National Research, Development and Innovation Office - NKFIH under the grants SNN 129364, FK 132060, and KKP-133819.

\end{document}